\documentclass[11pt]{amsart}
\usepackage{amsmath, amsfonts, amscd, latexsym, amsthm, amssymb}
\usepackage{hyperref}
\usepackage[margin=1.5in]{geometry}

\def \c{\mathbb{C}}
\def \z{\mathbb{Z}}
\def \r{\mathbb{R}}
\def \n{\mathbb{N}}
\def \p{\mathbb{P}}

\def \k{{\bf k}}

\def \.{\cdot}

\def \Reg{\textup{Reg}}

\def \Vol{\textup{Vol}}

\def \Sym{\textup{Sym}}

\def \GL{\textup{GL}}

\def \Reg{\textup{Reg}}

\def \End{\textup{End}}
\def \conv{\textup{conv}}

\def \GL{\textup{GL}}
\def \SL{\textup{SL}}

\theoremstyle{plain}
\newtheorem{Th}{Theorem}[section]
\newtheorem{Lem}[Th]{Lemma}
\newtheorem{Prop}[Th]{Proposition}
\newtheorem{Cor}[Th]{Corollary}

\theoremstyle{definition}
\newtheorem{Ex}[Th]{Example}
\newtheorem{Def}[Th]{Definition}
\newtheorem{Rem}[Th]{Remark}

\begin{document}
\title[Asymptotic enumeration of highest weights in tensor powers]
{A remark on asymptotic enumeration of highest weights in tensor powers of a representation}
\author{Kiumars Kaveh}
\address{Department of Mathematics, School of Arts and Sciences, University of Pittsburgh, 
301 Thackeray Hall, Pittsburgh, PA  15260, U.S.A.}
\email{kaveh@pitt.edu} 

\begin{abstract}
We consider the semigroup $S$ of highest weights appearing in tensor powers $V^{\otimes k}$ of
a finite dimensional representation $V$ of a connected reductive group. We describe the cone generated by $S$
as the cone over the weight polytope of $V$ intersected with the positive Weyl chamber.
%, i.e. the convex hull of the union of Weyl orbits of highest weights in $V$ intersected with the positive Weyl chamber.
From this we get a description for the asymptotic of the number of highest weights appearing in $V^{\otimes k}$ in terms of the volume of this polytope.
\end{abstract}



\maketitle

%\tableofcontents

\section{Introduction}
Let $G$ be a connected reductive algebraic group over an algebraically closed field $\k$ of characteristic $0$, and
let $V$ be a finite dimensional $G$-module. We consider the semigroup of dominant weights
$$S = S(V) = \{(k, \lambda) \mid V_\lambda \text{ appears in } V^{\otimes k}\},$$
where $V_\lambda$ is the irreducible representation with highest weight $\lambda$.
In this note we describe the cone $C(S)$ of this semigroup, i.e. the smallest closed convex cone (with apex at the
origin) containing $S$ (in other words, the closure of the convex hull of $S \cup \{0\}$). We use this to
describe the asymptotic of the number of highest weights $\lambda$ appearing in $V^{\otimes k}$.

This work is in the spirit of the general theory of semigroups of integral points and Newton-Okounkov bodies
developed in \cite{Askold-Kiumars-Newton-Okounkov} and \cite{Askold-Kiumars-affine}.

Let $A$ denote the finite set of highest weights in $V$, i.e. the dominant weights $\lambda$ where
$V_\lambda$ appears in $V$. Consider the union of all the Weyl group orbits of $\lambda \in A$ and let
$P^+(V)$ be its convex hull intersected with the positive Weyl chamber.
We show that the slice of the cone $C(S)$ at $k=1$ coincides with
the polytope $P^+(V)$ (Theorem \ref{th-main}). The main tool in the proof will be the PRV theorem on the 
tensor product of irreducible representations. Beside the PRV theorem The rest of arguments are elementary in nature.

Let $H_V(k)$ denote the number of dominant weights $\lambda$ where $V_\lambda$ appears in $V^{\otimes k}$.
From general statements about semigroups of integral points we then conclude that
$H_V(k)$ grows of degree $q = \dim(P^+(V))$, i.e. the limit
$$a_q = \lim_{k \to \infty} H_V(k)/k^q$$ exists and is non-zero. In addition $a_q$ is equal to the
(properly normalized) volume of the polytope $P^+(V)$.

In the last section we discuss the connection between the semigroup $S(V)$, its associated polytope
$P^+(V)$ and the moment polytope of $G$-varieties.

At the end, we would like to mention the related paper of Tate and Zeldtich \cite{Zelditch} in 
which the authors address the (more difficult and independent) question of describing 
the asymptotic behavior of multiplicities of irreducible representations
appearing in tensor powers of an irreducible representation.
This note is also related to \cite{KKh-Kazarnovskii} (we should point out that 
the first version of this note appeared in arXiv before \cite{KKh-Kazarnovskii}).

To make this note accessible to a wider range of audience we have tried to cover most of the background material.
\\

\noindent{\bf Notation:} Throughout the paper we will use the following notation. $G$ denotes
a connected reductive algebraic group over an algebraically closed field $\k$ of characteristic $0$.
\begin{itemize}
\item[-] We fix a Borel subgroup $B$ and a maximal torus $T$ in $G$. 
%$R = R(G, T)$ is the root system with $R^+ = R^+(G, T)$ the subset of positive
%roots for the choice of $B$. 
The Weyl group of $(G, T)$ is denoted by $W$. It contains a unique longest element denoted by $w_0$.
\item[-] $\Lambda$ denotes the weight lattice of $G$
(that is, the character group of $T$), and $\Lambda^+$ is the subset
of dominant weights (for the choice of $B$). Put $\Lambda_\r = \Lambda \otimes_{\z} \r$.
Then the convex cone generated by $\Lambda^+$ in $\Lambda_\r$ is the
positive Weyl chamber $\Lambda^+_{\r}$.
\item[-] For a weight $\lambda \in
\Lambda$, the irreducible $G$-module corresponding to $\lambda$ will
be denoted by $V_\lambda$ and a highest weight vector in $V_\lambda$ will
be denoted by $v_\lambda$. Finally for a dominant weight $\lambda$, we put
$\lambda^* = -w_0(\lambda)$ which is again a dominant weight. One has
$V_{\lambda^*} \cong V^*_\lambda$ as $G$-modules.
\end{itemize}

\section{Semigroups of integral points and convex bodies} \label{sec-semigroup}
Let $S \subset \n \times \z^n$ be a semigroup of integral points (i.e. $S$ is closed under addition).

Let $C(S)$ be the smallest closed convex cone (with apex at the origin) containing $S$. 
Also let $G(S)$ be the subgroup of $\z^{n+1}$ generated by $S$
and, $L(S)$ the linear subspace of $\r^{n+1}$ spanned by $S$. The sets $C(S)$ and $G(S)$ lie in $L(S)$.
To $S$ we associate its {\it regularization} which is the semigroup $\Reg(S) = C(S) \cap G(S)$.
The regularization $\Reg(S)$ is a simpler semigroup with more points which contains the semigroup
$S$. In \cite[Section 1.1]{Askold-Kiumars-Newton-Okounkov} it is proved that
{\it the regularization $\Reg(S)$ asymptotically approximates the semigroup $S$.}
More precisely:
\begin{Th}[Approximation Theorem] \label{th-approx}
Let $C' \subset C(S)$ be a strongly convex cone which intersects
the boundary (in the topology of the linear space $L(S)$) of the cone $C(S)$
only at the origin. Then there exists a constant $N>0$ (depending on $C'$)
such that each point in $G(S) \cap C'$ whose distance from the origin is bigger than $N$ belongs to $S$.
\end{Th}

Let $\pi: \r \times \r^n \to \r$ denote the projection on the first factor.
We call a semigroup $S \subset \n \times \z^n$ a {\it non-negative semigroup}
if it is not contained in the hyperplane $\pi^{-1}(0)$.
If in addition the cone $C(S)$ intersects the hyperplane
$\pi^{-1}(0)$ only at the origin,
$S$ is called a {\it strongly non-negative semigroup}.

Let $S_k = S \cap \pi^{-1}(k)$ be the set of points in $S$ at level $k$.
For simplicity throughout this section we assume that $S_1 \neq \emptyset$.

%Then $\pi(S)$ consists of $k$ such that $S_k \neq \{0\}$.
%It is a subsemigroup in $\n$.
%Let $m(S)$ be the index of the subgroup generated by $\pi(S)$ in $\z$. One shows that for sufficiently
%large $k$, we have $k \in \pi(S)$ if and only if $k$ is divisible by $m(S)$.
%\begin{Rem}
%The assumption $m(S) = 1$ is not crucial and one can slightly modify all the statements
%hat follow so that they hold without this assumption.
%end{Rem}
%As above let $C(S)$ be the
%smallest closed convex cone containing $S$, $G(S)$ the subgroup of
%$\z^{n+1} = \z \times \z^n$ generated by $S$, and $L(S)$ the rational subspace in
%$\r^{n+1}$ spanned by $S$.
%If in addition the cone $C(S)$ intersects the hyperplane
%$\pi^{-1}(0)$ only at the origin,
%$S$ is called a {\it strongly non-negative semigroup}.

We denote the group $G(S) \cap \pi^{-1}(0)$ by $\Lambda(S)$ and call it the
{\it lattice associated to the non-negative semigroup $S$}.
%The index of this sublattice in
%$\{0\} \times \z^n$ will be denoted by $\ind(\Lambda(S))$ or simply $\ind(S)$.
Finally, the number of points in $S_k$ is denoted by $H_S(k)$. $H_S$ is called the
{\it Hilbert function of the semigroup $S$}.

\begin{Def}[Newton-Okounkov convex set]
We call the projection of the convex set $C(S) \cap \pi^{-1}(1)$ to $\r^n$ (under the
projection on the second factor $(1, x) \mapsto x$),
the {\it Newton-Okounkov convex set of the semigroup $S$} and denote it by $\Delta(S)$.
In other words,
$$\Delta(S) = \overline{\conv(\bigcup_{k>0} \{x/k \mid (k, x) \in S_k\})}.$$
If $S$ is strongly non-negative then $\Delta(S)$ is compact and hence a convex body.
\end{Def}

%\begin{Rem}
%In \cite{Askold-Kiumars-Newton-Okounkov} the Newton-Okounkov convex set is defined as
%$C(S) \cap \pi^{-1}(1)$ (instead of its projection to $\r^n$). For the purposes of this paper
%we prefer to work with its projection.
%\end{Rem}

%Let us define the notion of volume normalized with respect to a lattice.
%\begin{Def}[Normalized volume] \label{def-int-volume}
Let $\Lambda \subset \r^n$ be a lattice of full rank $n$.
Let $E \subset \r^n$ be a subspace of dimension $q$ which is rational with respect to $\Lambda$.
The {\it Lebesgue measure normalized with respect to the lattice $\Lambda$} in $E$ is the
Lebesgue measure $d\gamma$ in $E$ normalized such that the smallest measure of a $q$-dimensional
parallelepiped with vertices in $E \cap \Lambda$ is equal to $1$. The measure of a subset
$A \subset E$ will be called its {\it normalized volume}
and denoted by $\Vol_q(A)$ (whenever the lattice $\Lambda$ is clear from the context).
%\end{Def}

Let $H_S$ and $H_{\Reg(S)}$ be the Hilbert functions of $S$ and its regularization respectively.
From Theorem \ref{th-approx} it follows that
$H_S(k)$ and $H_{\Reg(S)}(k)$ have the same asymptotic as $k$ goes to infinity.
Thus the Newton-Okounkov convex set $\Delta(S)$ is responsible for the
asymptotic behavior of the Hilbert function of $S$ (\cite[Section 1.4]{Askold-Kiumars-Newton-Okounkov}):

\begin{Th} \label{th-asymp-H_S-vol-Delta}
%Let $m=m(S)$.
The function $H_S(k)$ grows like $a_q k^q$ where $q$ is the dimension of the convex body $\Delta(S)$.
This means that the limit $$a_q = \lim_{k \to \infty} H_S(k)/k^q$$ exists and is non-zero. Moreover,
the {$q$-th growth coefficient $a_q$} is equal to $\Vol_q(\Delta(S))$, where the volume is normalized with respect to the lattice $\Lambda(S)$.
\end{Th}

Finally we make an observation which will be used later in proof of the main result (Theorem \ref{th-main}).
\begin{Prop} \label{prop-cone-S'}
Let $S \subset \n \times \z^n$ be a non-negative semigroup and $C = C(S)$ the cone
associated to $S$. Let $C' \subset C$ be a convex cone of full dimension (centered at the origin)
and $S' = S \cap C'$ the subsemigroup consisting of all the points of $S$ contained in $C'$.
Then the cone $C(S')$ associated to $S'$ coincides with $C'$.
\end{Prop}
\begin{proof}
Clearly $C(S') \subset C'$. By contradiction suppose
$C(S')$ is not equal to $C'$. Then there is a convex cone $\tilde{C} \subset C'$ of full dimension
which intersects $C(S')$ and the boundary of $C'$
(in the topology of the subspace $L(S)$)
only at the origin. Since $\tilde{C}$ has full dimension it contains a rational point (with respect to the lattice
$\Lambda(S)$) which then implis that it contains a point in $\Lambda(S)$. Now applying Theorem \ref{th-approx} we see that
the convex cone $\tilde{C}$ should contain a point in $S'$ which contradicts that $C(S') \cap \tilde{C} = \emptyset.$
\end{proof}

In the rest of the paper we will deal with semigroups and convex polytopes naturally
associated to a reductive group $G$ and its representations.

\begin{Rem}
The proof of Theorem \ref{th-approx} relies on the proof of the special case when $S$ is a finitely generated semigroup. 
The semigroups appearing in this note turn out to be in fact finitely generated, although we will not use this fact.
\end{Rem}

\section{Main result} \label{sec-main}
Let $V$ be a finite dimensional $G$-module.
Define the set $S(V) \subset \n \times \Lambda^+$ by
$$S(V) = \{ (k, \lambda) \mid V_\lambda \text{ appears in } V^{\otimes k} \}.$$

If $v_\lambda$ and $v_\mu$ are highest weight vectors in $V^{\otimes k}$ and $V^{\otimes \ell}$
with weights $\lambda$ and $\mu$ respectively, then $v_\lambda \otimes v_\mu$ is a highest weight vector
in $V^{\otimes k + \ell}$ of weight $\lambda+\mu$. It follows that $S(V)$ is a semigroup with respect to addition.
Let $\Delta(V)$ denote the Newton-Okounkov body of the semigroup $S(V)$. In other words,
$$\Delta(V) = \overline{\conv(\bigcup_{k>0}\{ \lambda/k \mid V_\lambda \text{ appears in } V^{\otimes k}\}}).$$

Also let $A$ be the collection of $\gamma$ where $V_\gamma$ appears in $V$. 
\begin{Def}
The {\it weight polytope of $V$}
is defined by $P(V) = \conv\{w(\gamma) \mid w \in W,~ \gamma \in A\}$, i.e. the convex hull of the union of 
Weyl orbits of $\gamma \in A$. We will denote the intersection of $P(V)$ with the positive Weyl chamber $\Lambda^+_\r$ 
by $P^+(V)$ and call it the {\it moment polytope of $V$}.
\end{Def}

\begin{Th} \label{th-main}
$\Delta(V)$ coincides with the moment polytope $P^+(V)$.
\end{Th}

We will need the following well-known fact:
\begin{Lem} \label{lem-product-irr-rep-weight-polytope}
Let $\lambda_1, \lambda_2$ be dominant weights and let $V_\gamma$ appear in
$V_{\lambda_1} \otimes V_{\lambda_2}$. Then $\gamma = \lambda_1 + \lambda_2 - \sum_{\alpha \in R^+} c_\alpha \alpha$
where $c_\alpha \geq 0$ ($R^+$ denotes the set of positive roots). From this it follows that
$\gamma$ belongs to the convex hull of the Weyl orbit of $\lambda_1 + \lambda_2$.
\end{Lem}

Our tool to prove Theorem \ref{th-main} is the well-known PRV-Kumar theorem regarding the
tensor product of two irreducible representations. It was conjectured by
K. Parthasarathy, R. Ranga Rao and V. Varadarajan in \cite{PRV}.
Later it was proved by S. Kumar in \cite{Kumar}. We briefly recall its statement.

\begin{Th}[PRV-Kumar theorem] \label{th-PRV}
Let $\lambda_1, \lambda_2 \in \Lambda^+$ be two dominant weights. Suppose for
two Weyl group elements $w_1, w_2 \in W$ we have $\gamma = w_1(\lambda_1) + w_2(\lambda_2)$ is a dominant weight.
Then $V_\gamma$ appears in the decomposition of the tensor product $V_{\lambda_1} \otimes V_{\lambda_2}$ into
irreducible representations.
\end{Th}

Define the set $\tilde{S}(V) \subset \n \times \Lambda$ by:
$$\tilde{S}(V) = \{ (k, w(\lambda)) \mid w \in W,~ V_\lambda \text{ appear in } V^{\otimes k}\}.$$
Roughly, speaking $\tilde{S}(V)$ is the union of $W$-otbits of $\lambda$ for which $V_\lambda$ appears in 
some tensor power $V^{\otimes k}$. Notice that $S(V) = \tilde{S}(V) \cap (\n \times \Lambda^+)$.
The following is a straight forward corollary of Theorem \ref{th-PRV}
\begin{Cor} \label{cor-PRV}
1) $\tilde{S}(V)$ is a semigroup. 2) The convex body $\Delta(\tilde{S}(V))$ associated to 
this semigroup coincides with $P(V)$. 
\end{Cor}
\begin{proof}
1) Let $(k, w_1(\lambda_1))$, $(\ell, w_2(\lambda_2))$ be two elements in $\tilde{S}(V)$.
We can write $w_1(\lambda_1) + w_2(\lambda_2)$ as $w(\lambda)$ for some $\lambda \in \Lambda^+$, $w \in W$.
By Theorem \ref{th-PRV}, $V_\lambda$ appears in $V_{\lambda_1} \otimes V_{\lambda_2}$ and hence 
it appears in $V^{\otimes \ell+k}$. This shows that $(\ell+k, w_1(\lambda_1)+w_2(\lambda_2)) = 
(\ell+k, w(\lambda))$ belongs to $\tilde{S}(V)$ which proves 1). 2) 
Since $\tilde{S}(V)$ is a semigroup and $P(V)$ is by definition the convex hull of $\tilde{S}(V)_1$, it follows that 
$kP(V)$ is contained in the convex hull of $\tilde{S}(V)_k$.  
On the other hand, by Lemma \ref{lem-product-irr-rep-weight-polytope}, for any integer $k>0$, 
the convex hull of $\tilde{S}(V)_k$ is contained in
$kP(V)$ and hence $\Delta(\tilde{S}(V))$ coincides with $P(V)$. 
\end{proof}

\begin{proof}[Proof of Theorem \ref{th-main}]
From Proposition \ref{prop-cone-S'}, the convex body $\Delta(V)$ associated 
to the semigroup $S(V) \subset \tilde{S}(V)$ is just the intersection of 
$\Delta(\tilde{S}(V))$ with $\Lambda^+_\r$. By Corollary \ref{cor-PRV}(2)
we know that $\Delta(\tilde{S}(V))$ is the weight polytope $P(V)$ which finishes the proof.
\end{proof}

\begin{Cor} \label{cor-main}
Let $H_V(k)$ be the number of $\lambda$ such that $V_\lambda$ appears in $V^{\otimes k}$.
Then $H_V(k)$ grows of degree $q = \dim P^+(V)$.  That is, the limit
$$a_q = \lim_{k \to \infty} H_V(k) / k^q$$ exists and is non-zero. Moreover, $a_q$
is equal to $\Vol_q(P^+(V))$, where
volume is the Lebesgue measure in $\Lambda_\r$ normalized with respect to the
lattice $\Lambda(S(V)) \subset \Lambda$.
\end{Cor}
\begin{proof}
Follows directly from Theorem \ref{th-main} and Theorem \ref{th-asymp-H_S-vol-Delta}.
\end{proof}

\begin{Ex} \label{ex-SL_2}
Perhaps the simplest case of Theorem \ref{th-main} and Corollary \ref{cor-main} is $G = \SL(2, \c)$. One knows that the irreducible 
representations of $\SL(2, \c)$ are enumerated by $n \in \z_{\geq 0}$ as $V_n = \Sym^n(\c^2)$, where $\SL(2, \c)$ acts on $\c^2$ in the usual way.
It is well-known that for $n, m \geq 0$:
\begin{equation} \label{equ-SL_2-tensor-product}
V_n \otimes V_m = V_{|n-m|} \oplus V_{|n-m+2|} \oplus \cdots \oplus V_{n+m}.
\end{equation}
Let $V = m_1V_{n_1} \oplus \cdots \oplus m_rV_{n_r}$ be the decomposition of a finite dimensional representation of $\SL(2, \c)$ into irreducible representations
where $0 \leq n_1 < \cdots < n_r$. 
From \eqref{equ-SL_2-tensor-product} we see that $\Delta(V) = [0, n_r]$. Corollary \ref{cor-main} then states that the number of irreducible representations
appearing in $V^{\otimes k}$ is asymptotically equal to $kn_r$. For $G = \SL(n, \c)$ the decomposition of tensor product of two irreducible representations $V_\lambda \otimes V_\gamma$ into irreducible representations 
is more complicated. Exactly what irreducible representations appear in $V_\lambda \otimes V_\gamma$ 
is related to the so-called Horn's conjecture/theorem (see for example \cite[Section 3]{Fulton}). 
\end{Ex}

\section{Relation with moment polytope of group actions} \label{sec-moment}
In this section we see how the moment polytope $P^+(V)$ appears as a moment polytope for
the action of $G \times G$ on $G$.

Let $V$ be a finite dimensional $G$-module, and $X \subset \p(V)$ an
irreducible closed $G$-invariant subvariety. Let $R = \bigoplus_{k \geq 0} R_k$ denote the
homogeneous coordinate ring of $X$. It is a graded $G$-algebra. Following Brion, one defines the
{\it moment polytope $\Delta(X)$} to be:
$$\Delta(X) = \overline{\conv(\bigcup_{k > 0}\{\lambda / k \mid V_\lambda \text{ appears in } R_k \})}.$$
One shows that $\Delta(X) \subset \Lambda^+_\r$ is a polytope (see \cite{Brion-moment}).
Moreover, when $\k = \c$ and $X$ is smooth, the polytope $\Delta(X)$ can be identified with the moment polytope of
$X$ regarded as a Hamiltonian space for the action of a maximal compact subgroup $K$ of $G$ and
the symplectic structure induced from the projective space (see for example \cite{G-S}).

Let $G \times G$ act on $G$ via multiplication from left and right. Let $\k[G]$ denote the
algebra of regular functions on the variety $G$. It is a rational $(G \times G)$-module.
It is well-known that for each dominant weight $\lambda$, the $(\lambda^*, \lambda)$-isotypic component
$\k[G]_{(\lambda^*, \lambda)}$
is isomorphic to $V_{\lambda^*} \otimes V_{\lambda}$. Moreover, any isotypic component of
$\k[G]$ is of this form for some $\lambda$.
In fact any $(G \times G)$-isotypic component
$\k[G]_{(\lambda^*, \lambda)}$ in $\k[G]$ is the linear span of the matrix entries corresponding to
the representation of $G$ in $V_\lambda$ (see \cite{Kraft}). Now if
$\lambda_1, \lambda_2 \in \Lambda^+$ are two dominant weights, the product
$\k[G]_{(\lambda^*_1,\lambda_1)}\k[G]_{(\lambda^*_2,\lambda_2)}$ is the
linear span of the matrix entries corresponding to $V_{\lambda_1} \otimes V_{\lambda_2}$. This
shows that we have the following decomposition for the product of isotypic components in the
algebra $\k[G]$:
\begin{equation} \label{equ-k[G]-product-isotypic}
\k[G]_{(\lambda_1^*, \lambda_1)} \k[G]_{(\lambda_2^*, \lambda_2)} =
\bigoplus_{\gamma \in \chi(\lambda_1, \lambda_2)} V_{\gamma^*} \otimes V_{\gamma},
\end{equation}
where $\chi(\lambda_1, \lambda_2)$ denotes the collection of all
$\gamma \in \Lambda^+$ for which $V_\gamma$ appears in $V_{\lambda_1} \otimes V_{\lambda_2}$.

Now let $\pi: G \to \GL(V)$ be a finite dimensional representation.
Then $\End(V)$ is naturally a $(G \times G)$-module where $G \times G$ acts via $\pi$ by
multiplication from left and right. Let $\tilde{\pi}: G \to \p(\End(V))$ be the induced map to
projective space and let $X$ be the closure of the image of $G$ in $\p(\End(V))$.
It is a $(G \times G)$-invariant closed irreducible subvariety.

From (\ref{equ-k[G]-product-isotypic}) one can see the following:
\begin{Prop} \label{prop-homog-coor-G-tensor-powers}
Let $R = \bigoplus_{k} R_k$ denote the homogeneous coordinate ring of $X$ in $\p(\End(V))$.
Then for $k > 0$ we have: $V_{(\lambda^*, \lambda)}$ appears in $R_k$ if and only if $V_{\lambda}$ 
appears in $V^{\otimes k}$. It follows that, under the projection on the second factor, 
$\Delta(X) \subset \Lambda^+_\r \times \Lambda^+_\r$ identifies with $P^+(V)$.
\end{Prop}

\begin{Rem}
The relation between the moment polytope of $X$ (i.e. a group compactification) and 
the polytope $P^+(V)$ has also been shown in \cite{Kazarnovskii} using methods from symplectic geometry.
\end{Rem}

%\begin{Rem}
%We conjecture that PRV-Kumar theorem holds for the coordinate ring of a
%symmetric variety $G/H$.
%\end{Rem}

\noindent{\bf Acknowledgement:} the author would like to thank Askold Khovanskii, Alexander Yong and Kevin Purbhoo for helpful discussions,
as well as Shrawan Kumar for helpful email correspondence.\\

\end{document}